\documentclass[12pt,a4paper]{amsart}

\usepackage{amssymb}
\usepackage{graphicx}
\usepackage{tikz}
\usetikzlibrary{calc}
\usepackage{lineno}

\theoremstyle{plain}
\newtheorem{theorem}{Theorem}[section]
\newtheorem{corollary}[theorem]{Corollary}
\newtheorem{lemma}[theorem]{Lemma}

\theoremstyle{definition}

 \theoremstyle{remark}
\newtheorem{remark}{Remark}[section]
\newtheorem{example}{Example}[section]

\numberwithin{equation}{section}

\numberwithin{table}{section}

\numberwithin{figure}{section}

\begin{document}

\title[Notes on Hamiltonian threshold and chain graphs]{Notes on Hamiltonian threshold and chain graphs}

\author{Milica An\dj eli\'c}
\address{Department of Mathematics, Kuwait University, Safat 13060,
Kuwait} \email{milica.andelic@ku.edu.kw}

\author{Tamara Koledin}
\address{Faculty of Electrical Engineering, University of Belgrade, Bulevar kralja Aleksandra 73,
	Belgrade, Serbia} \email{tamara@etf.rs}

\author{Zoran Stani\' c}
\address{Faculty of Mathematics, University of Belgrade, Studentski trg 16,
	Belgrade, Serbia} \email{zstanic@math.rs}


\date{\today}


\begin{abstract}
We revisit results obtained in [F. Harary, U. Peled, Hamiltonian threshold graphs, Discrete Appl.~Math., 16 (1987), 11--15], where several necessary and necessary and sufficient conditions for a connected threshold graph to be Hamiltonian were obtained.  We present these results in new forms, now stated in terms of structural parameters that uniquely define the  threshold graph and we extend them to chain graphs.  We also identify the chain graph with minimum number of Hamilton cycles within the class of Hamiltonian chain graphs of a given order. 
\end{abstract}

\maketitle \medskip\noindent

\section{Introduction}

\noindent  A threshold graph can be defined in many ways, as can be seen in \cite{MP95}.
Here we follow the definition  via binary generating sequences.  Accordingly, a \textit{threshold graph} $G(b)$ is obtained from its \textit{binary generating sequence} of the form  $b=(b_1 b_2 \cdots b_n)$ in the following way:

\begin{itemize}
	\item[(i)] for $i=1$, $G_1 = G(b_1) = K_1$, i.e., a single vertex;
	\item[(ii)] for $i\geq 2$, with $G_{i-1} = G(b_1 b_2 \cdots b_{i-1})$ already constructed,
	$G_i = G(b_1 b_2 \cdots b_{i-1} b_{i})$ is formed by adding an isolated vertex to $G_{i-1}$ if $b_i = 0$ (that is, a vertex non-adjacent to any vertex in $G_{i-1}$) or by adding a dominating vertex to $G_{i-1}$ if $b_i=1$ (that is, a vertex adjacent to all the vertices in $G_{i-1}$).
\end{itemize}

Clearly, $G(b) = G_{n}$. A schematic representation of a threshold graph is illustrated in Fig~\ref{fig:nested}(a); its vertices are partitioned into cells $U_i, V_i~(1\leq i\leq h)$.

\begin{figure}[h]
	\centering
	\includegraphics[width=140mm,angle=0]{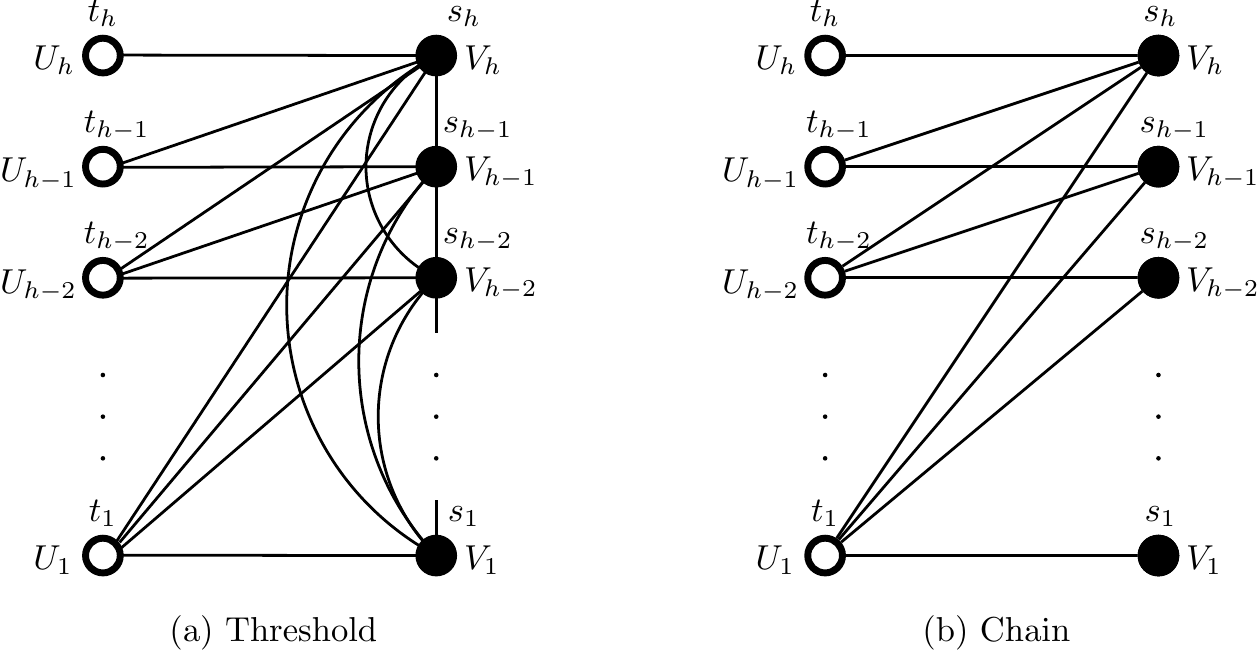}
	\caption{A schematic representation of a threshold and a chain graph.}\label{fig:nested}
\end{figure}

The vertices in $U=\bigcup_{i=1}^h U_i$ induce a co-clique, while the vertices in 
 $V=\bigcup_{i=1}^h V_i$ induce a clique. An equivalent definition of a threshold graph says that a graph is a threshold graph if it does not contain neither of $2K_2$ (the disjoint union of two edges) $P_4$ (the 4-vertex path) or $C_4$ (the 4-vertex cycle) as an induced subgraph \cite{StaM}.
 
 A bipartite counterpart to the threshold graph is a \textit{chain graph} which is generated by the same binary sequence in the following way:
\begin{itemize}
\item[(i)] for $i=1$, $G_1=G(b_1)=K_1$, i.e., a single  vertex belonging to one colour class, say white vertex;
\item[(ii)] for $i\geq 2$, with $G_{i-1}=G(b_1 b_2 \cdots b_{i-1})$, $G_i=G(b_1 b_2 \cdots b_{i-1} b_i)$ is obtained by adding to $G_{i-1}$ an isolated white vertex if $b_{i}=0$ or a black vertex which dominates all previously added white vertices if $b_i=1$.
\end{itemize}

A schematic representation is illustrated in Figure~\ref{fig:nested}(b). 

Here are some conventions on binary generating sequences. First, observing that, due to the defining rule~(i),  (either a threshold or a chain) graph is independent of $b_1$, we use the convention that the binary sequences always start with zero. Moreover, if $b_n = 1$, then the corresponding graph is connected; otherwise, it is connected up to isolated vertices. Accordingly, since we restricted ourselves to connected graphs, a binary sequence can be written as \begin{equation}\label{eq:bsequence}
 b=(0^{t_1}1^{s_1})(0^{t_2}1^{s_2})\cdots (0^{t_h}1^{s_h}),~~\text{where}~~ t_i, s_i>0,~~ \text{for}~~ 1\leq i\leq h.\end{equation}
 (Naturally, $t_i$'s and $s_i$'s are  lengths of maximum runs of consecutive zeros and ones, respectively.)

In a so-called \textit{split graph} the vertex set can be divided into two disjunct sets, say $U$ and~$V$, in such a way that $U$ induces a co-clique and $V$ induces a clique. Evidently, every threshold graph is a split graph and the neighbourhoods of the vertices are totally ordered by inclusion. By deleting all the edges that belong to the clique $V=\bigcup_{i=1}^h V_i$ (see Figure \ref{fig:nested}) of a threshold graph, we obtain the chain graph that is generated by the same binary sequence. {Note that for any $x\in U_1$, the subgraph induced by $V\cup \{x\}$  induces a maximal clique in the corresponding threshold graph.}

We say that a graph is \textit{Hamiltonian} if it contains a cycle passing through all of its vertices. 
Every such a cycle is called a \textit{Hamilton cycle}. 

In this paper we revisit the results obtained in \cite{HP87} on Hamiltonicity of threshold  graphs. We give necessary and sufficient conditions for a threshold graph to be Hamiltonian in terms of its  generating binary sequence.

The paper is organized as follows. In Section~\ref{non} we recall the results from \cite{HP87}. In Section~\ref{rev} we interpret these results in terms of the entries of the generating binary sequence of a threshold graph and give a criterion  for  Hamiltonicity of a threshold graph that can be deduced directly from its binary sequence. This criterion is implemented in the algorithm presented in Section~\ref{alg}, where we also include an algorithm that determines whether a given chain graph is Hamiltonian. In Section \ref{min} we  identify  the chain graph of a given  order that contains minimum number of Hamilton cycles. Some concluding notes and directions for future research are given in Section~\ref{con}.

\section{Results obtained in \cite{HP87}}\label{non}

\noindent Let $G$ be a threshold graph with vertex set $I\cup J$, where $I$ (with $|I|=r$) induces a co-clique and $J$ (with $|J|=s$) induces a maximal clique. Let further $B$ denote the chain graph obtained from $G$ by deleting all edges in the subgraph induced by $J$. For $B$, let $d_1\leq d_2\leq\cdots\leq d_r$ and $e_1\leq e_2\leq\cdots\leq e_s$ denote the degrees of the vertices $x_1, x_2, \ldots, x_r\in I$ and $y_1, y_2, \ldots, y_s\in J$, respectively. 

In order to determine whether a given threshold graph is Hamiltonian, the authors of~\cite{HP87} first showed that  this problem can be reduced to the question of Hamiltonicity of the corresponding chain graph $B$ with the colour classes of the same size, as shown in the sequel.

The first lemma gives sufficient conditions for a split graph to be non-Hamiltonian.

\begin{lemma}(\cite{HP87})\label{lem1}
Let $G$ be a split graph with vertex set $I\cup J$, where $I$ induces a co-clique of  size $r$ and $J$ induces a maximal clique of  size $s$. If either $r>s$ or $r<s$ with $e_{s-r}=0$, then  $G$ is not Hamiltonian.
\end{lemma}

In what follows we only consider threshold graphs with $r\geq 2$, since for $r=0$, $G$ is Hamiltonian if and only if $s\geq 3$, while for $r=1$, $G$ is Hamiltonian if and only if $d_1\geq 2$.

Since any threshold graph is a split graph, by the previous lemma,  Hamiltonian threshold graphs satisfy  $2\leq r\leq s$ and $e_{s-r}>0$. The next lemma shows that the problem under the consideration can be reduced to the Hamiltonicity of threshold graphs with  {$|U|=|V|$.} 

\begin{lemma}(\cite{HP87})\label{lem2}
If  $2\leq r<s$ and $e_{s-r}>0$, then the threshold graph $G$ is  Hamiltonian if and only if the threshold subgraph $G^*$ obtained by deleting the vertices $y_1, y_2, \ldots, y_{s-r}$ is Hamiltonian.
\end{lemma}

\begin{remark}

In \cite{HP87} the  conclusion that after dropping $y_1, y_2, \ldots, y_{s-r}$ from $J$ the resulting threshold graph will have a maximal clique and co-clique of the same size is wrong. The correct conclusion is that in this case $s=r+2$. In the resulting graph a clique induced by $\{y_{s-r+1},\ldots, y_s\}$ is not a maximal one. A maximal one can be obtained by adding a vertex from a co-clique of the smallest degree. For example, in the threshold graph generated by  $0^21^20^21^3$, the size of a maximal clique is $6$ and the size of a co-clique is $3$. After dropping $y_1, y_2, y_3$
we obtain the threshold graph generated by $0^3 1^3$. However, in this graph $r\neq s$, since $r=2$ and $s=4$.

\end{remark}

In what follows we assume that in a threshold graph $G$, $|U|=|V|$. Then the edges in the clique cannot be used in any Hamiltonian cycle, and therefore can be dropped from $G$, yielding the chain graph $B$ with $|U|=|V|\geq 2$.

For $q=0,1,\ldots,|U|-1$ denote by $S_q$  the set of inequalities 
\begin{align*}
d_j&\geq j+1, \quad j=1,2,\ldots, q;\\
e_j&\geq j+1, \quad j=1,2,\ldots,|U|-1- q.
\end{align*}

The next theorem gives two equivalent conditions for a chain graph $B$ with $|U|=|V|\geq 2$ to be Hamiltonian.

\begin{theorem}(\cite{HP87})\label{th_Sq}
If $|U|=|V|\geq 2$, then the following conditions are equivalent:
\begin {itemize}
\item[(a)] $B$ is Hamiltonian;
\item[(b)] $S_q$ holds for some $q\in\{0,1,\ldots,r-1\}$; 
\item[(c)] $S_q$ holds for each $q\in\{0,1,\ldots,r-1\}$. 
\end{itemize}
\end{theorem}

To conclude, in order to determine whether an arbitrary  threshold graph  is Hamiltonian, all  three results reported in Lemma \ref{lem1}, Lemma \ref{lem2} and Theorem \ref{th_Sq} should be employed.

\section{New versions of results obtained in \cite{HP87}}\label{rev}

\noindent In this section we restate the results from Section~\ref{non} in terms of the entries of the generating binary sequence of a given threshold graph. Afterwards, we amalgamate them to obtain  a  result that gives necessary and sufficient conditions for a threshold graph  to be Hamiltonian.

Let $G$ and  $B$ be a threshold graph and a chain graph generated by a binary sequence~\eqref{eq:bsequence}. If $t_1\neq 1$, then the degrees of vertices in $B$, corresponding to {a co-clique $I$ and a maximal clique $J=V\cup\{x\}$ for some $x\in U_1$ are:}
\begin{eqnarray}\label{degU}
&d_1^{t_h},d_2^{t_{h-1}}, \ldots, d_h^{t_1-1}, ~ \text{with}~~ d_j=\sum_{i=h+1-j}^{h} s_i\\
&0^1,e_1^{s_1},e_2^{s_2}, \ldots, e_h^{s_h}, ~ \text{with}~~ {e_j=\sum_{i=1}^{j} t_i-1}.\label{degV}
\end{eqnarray}

Note that the vertex degrees are given in non-decreasing order, and according  to Figure~\ref{fig:nested}, they are the degrees of vertices in { $U_h, U_{h-1}, \ldots, U_1\setminus\{x\}$ and $\{x\}, V_1, V_2, \ldots, V_h$}, respectively.

Otherwise, if $t_1=1$, then  $G$ is a split graph in which the subgraph  induced by $V\cup U_1$ gives a maximal clique. Then the degrees of vertices in $B$ corresponding to the colour classes $\bigcup_{i=2}^h U_i$ and $V\cup U_1$ are:
\begin{eqnarray}\label{degU1}
&d_1^{t_h},d_2^{t_{h-1}}, \ldots, d_{h-1}^{t_2}, ~\text{with}~~ d_j=\sum_{i=h+1-j}^{h} s_i,\\
&0^{1+s_1}, e_2^{s_2},\ldots, e_{h}^{s_h}, ~\text{with}~~ e_j=\sum_{i=2}^{j} t_i.\label{degV1}
\end{eqnarray}

Let $T=\sum_{i=1}^h t_i$ and $S=\sum_{i=1}^h s_i$. From the previous observations, it follows that the size of the maximal clique $s$  and the size of the corresponding co-clique $r$  satisfy
$
r=
T-1$  and
$s=
S+1.$

We first state the following lemma that determines when $e_{s-r}=0$ may occur.

\begin{lemma}\label{zero}
Let $G$ be a threshold graph generated by \eqref{eq:bsequence}, such that $s>r$. Then $e_{s-r}=0$ holds if and only if  {$t_1\neq 1$ and $s-r=1$ or} $t_1=1$ and $s-r\in\{1, \ldots, s_1+1\}$. \end{lemma}
{\begin{proof}
From \eqref{degV}, it follows that for $t_1\neq 1$, $e_i=0$ if and only if $i=1$. Otherwise, from \eqref{degV1} $e_i=0$ if and only if $i\in\{1, \ldots, s_1+1\}$. This completed the proof. 
\end{proof}}

Next, Lemma \ref{lem1} applied to threshold graphs states the following.

\begin{lemma}
Let $G$ be a threshold graph generated by \eqref{eq:bsequence}. If $r>s$  or {$t_1\neq 1$ and $s-r=1$ or} $t_1=1$ and  $s-r\in\{1,2, \ldots, s_1+1\}$, then $G$ is not Hamiltonian.
\end{lemma}

In the sequel  we consider only threshold graphs, with {$s\geq r+2$ if $t_1\neq 1$} and with $s\geq r+s_1+2$ if $t_1=1$. 
Let the integer $\ell$ be defined in the following way: if {$s-r-1<s_1$}, then $\ell=0$; otherwise, $\ell$ is the least integer, such that {$\sum_{i=1}^\ell s_i\leq s-r-1<\sum_{i=1}^{\ell+1} s_i$ }(obviously, such an integer exists). 
 {The dropping of the vertices $y_1,y_2,\ldots, y_{s-r}$ from $J$, is equivalent to dropping  $y_1$ from $U_1$ and $y_2,\ldots, y_{s-r}$ from $V$. Consequently  the new graph $G^*$ is generated by the  binary sequence
\begin{equation}\label{G*}
\big(0^{\sum_{i=1}^{\ell+1}t_i-1}1^{\sum_{i=1}^{\ell+1} s_i-(s-r-1)}0^{t_{\ell+2}}1^{s_{\ell+2}}\cdots 0^{t_h}1^{s_h}\big).\end{equation}}
Next we state a reformulation of Lemma \ref{lem2}.
\begin{lemma}
Let $G$ be a threshold graph generated by \eqref{eq:bsequence} with $s\geq r\geq 2$. Then $G$ is Hamiltonian if and only if a threshold graph $G^*$ generated by \eqref{G*} is Hamiltonian.
\end{lemma}

For {$s^*_{\ell+1}=\sum_{i=1}^{\ell+1} s_i-(s-r-1)$}, the degrees of vertices in $U^*$ and $V^*$ in the corresponding bipartite graph $B^*$ of~$G^*$ given in non-decreasing order are
\begin{align*}
d_1^*= \cdots =d_{t_h}^*&=s_h,\\
d_{t_h+1}^*=\cdots=d_{t_h+t_{h-1}}^*&=s_{h-1}+s_h,\\
&\phantom{a}\vdots\\
d_{t_{\ell+3}+\cdots+t_{h}+1}^*=\cdots=d_{t_{\ell+2}+\cdots+t_{h}}^*&=\sum_{i=\ell+2}^h s_i,\\
d_{t_{\ell+2}+\cdots+t_{h}+1}^*=\cdots=d_{t_1+t_2+\cdots+t_{h}}^*&= {\sum_{i=1}^h t_i-1}\\
\mbox{and}\\
e_1^*= \cdots =e_{s^*_{\ell+1}}^*&= {\sum_{i=1}^{\ell+1} t_i-1},\\
e_{s^*_{\ell+1}+1}^*=\cdots=e_{s^*_{\ell+1}+s_{\ell+2}}^*&= {\sum_{i=1}^{\ell+2} t_i-1},\\
&\phantom{a}\vdots\\
e_{s^*_{\ell+1}+s_{\ell+2}+\cdots+s_{h-1}+1}^*=\cdots=e_{s^*_{\ell+1}+s_{\ell+2}+\cdots+s_{h}}^*&={\sum_{i=1}^h t_i-1}.
\end{align*}

Next, we consider the system of inequalities $S_q$, $q\in\{0,1,\ldots, T-1\}$, for   {$\sum_{i=1}^hs_i=\sum_{i=1}^h t_i$}.

 {\begin{lemma}
Let $G$ be a threshold graph generated by \eqref{eq:bsequence}, with $\sum_{i=1}^hs_i=\sum_{i=1}^h t_i\geq 2$. Then $G$ is Hamiltonian if and only  if
\begin{align}
\sum_{i=j}^h s_i&\geq \sum_{i=j}^h t_i+1,\quad \mbox{for} \quad j=2,3,\ldots, h\label{con1}.
\end{align}
\end{lemma}}

\begin{proof}
Recall from Section~\ref{non} that a threshold graph $G$, with {$S=T$}, is Hamiltonian if and only if the corresponding chain graph $B$ is Hamiltonian. Next, by Theorem \ref{th_Sq}, the chain graph $B$ is Hamiltonian if and only $S_q$ holds for $q=T-1$. On the other hand, for \eqref{con1} and each repeated vertex degree, we have 
\begin{eqnarray*}
\quad\begin{cases}
 \sum_{i=j}^h s_i\geq \sum_{i=j+1}^h t_i+2, \\
\qquad\phantom{a,,,}\vdots\\
\sum_{i=j}^h s_i\geq \sum_{i=j+1}^h t_i+t_j+1=\sum_{i=j}^h t_i+1.
\end{cases}
\end{eqnarray*}
Now, it is easy to see that  $S_{T-1}$ holds if and only if \eqref{con1} holds.
Note that the inequality $\sum_{i=1}^h s_i\geq \sum_{i=1}^h t_i-1+1$ is not included, since (by the assumption that $S=T$) it holds as equality. 
\end{proof}

\begin{remark}
The left hand sides of \eqref{con1} are equal to  the vertex degrees, while the right hand sides register the position of the last occurrence of the corresponding vertex degree augmented by $1$.  
\end{remark}

%

Gathering all the previous results, we arrive at our main result, the criterion for the Hamiltonicity of a threshold graph based on its generating binary sequence.

{\begin{theorem}\label{iff}
Let $G$ be a threshold graph generated by \eqref{eq:bsequence}, such that $r\geq 2$. If either $s\leq r+1$ for $t_1\neq 1$ or $s\leq r+s_1+1$ for $t_1=1$, then $G$ is not Hamiltonian. Otherwise, $G$ is Hamiltonian if and only if 
 $\ell+1=h$ or
\begin{align}
 \sum_{i=j}^h s_i&\geq \sum_{i=j}^h t_i, \quad \mbox{for}\quad j=\ell+2,\ell+3,\ldots, h,\label{conf1}
\end{align}
where for $s-r-1<s_1$, $\ell=0$, and otherwise $\ell$ is the least integer such that $\sum_{i=1}^\ell s_i\leq s-r-1<\sum_{i=1}^{\ell+1} s_i$.
\end{theorem}}

\begin{proof}
A threshold graph $G$ generated by \eqref{eq:bsequence} is Hamiltonian if and only $G^*$ generated by \eqref{G*} is Hamiltonian. Now, $G^*$ is Hamiltonian if and only if the corresponding chain graph $B^*$ generated by \eqref{G*} is Hamiltonian. The last one holds if and only if $\ell+1=h$ or otherwise if and only if the system of inequalities \eqref{con1}  holds for $B^*$, i.e., if and only if
\begin{align*}
\sum_{i=j}^h s_i&\geq{ \sum_{i=j}^h t_i, \quad \mbox{for}\quad j=\ell+2,\ell+3,\ldots, h,}\\
\end{align*}
which completes the proof.
\end{proof}

As a corollary we state a necessary and sufficient condition for a chain graph to be Hamiltonian. Note that Hamiltonian chain graph has the colour classes of the same size  (see \cite{bip}) and  cannot have any pendant edges. Moreover,  in any Hamiltonian chain graph  generated by \eqref{eq:bsequence}, the inequalities  $t_1\geq s_1+1$ and $s_h\geq t_h+1$ must hold. 

\begin{corollary}\label{hcg}
Let $G$ be a chain graph generated by \eqref{eq:bsequence}, such that $T\geq 2$. If either $S\neq T$ or $S=T$ and $t_1< s_1+1$ or $s_h<t_h+1$, then $G$ is not Hamiltonian. Otherwise, $G$ is Hamiltonian if and only if 
$$\sum_{i=j}^h s_i\geq \sum_{i=j}^h t_i+1 \quad \mbox{for}\quad j=2,3,\ldots, h.$$
\end{corollary}

\section{Algorithms}\label{alg}

\noindent In this section we  present  algorithms for recognition of Hamiltonian threshold graph and Hamiltonian chain graph. The input is  a binary generating sequence, and in return we obtain the decision whether the corresponding threshold (resp.~chain) graph is Hamiltonian or not.

\medskip

\noindent\textbf{Algorithm 1 (checks if a given threshold graph is Hamiltonian).}
	
\begin{enumerate}
\item[(0)] INPUT: Generating binary sequence $(0^{t_1}1^{s_1})(0^{t_2}1^{s_2})\cdots (0^{t_h}1^{s_h})$.
{\item Calculate $r$ and $s$. $r=\sum_{i=1}^h t_i-1$; $s=\sum_{i=1}^h s_i+1$.
\item If 
$r=1\land d_1=s_h\geq 2$ then RETURN TRUE. If $r=1\land d_1=s_h=1$ then 
RETURN FALSE.

\item If $(t_1\neq 1\land s\leq r+1)\lor(t_1=1\land s\leq 
r+s_1+1)$ then RETURN FALSE.

\item Determine the least integer $\ell$, such that $\sum_{i=1}^\ell s_i\leq s-r-1<\sum_{i=1}^{\ell+1} s_i$. If $s-r-1<s_1$, take $\ell=0$.
\item If $\ell+1=h$ or all inequalities in \eqref{conf1} hold  then RETURN TRUE. Otherwise, RETURN FALSE. }
\end{enumerate}

\medskip

\noindent\textbf{Algorithm 2 (checks if a given chain graph is Hamiltonian).}
	
\begin{enumerate}
\item[(0)] INPUT: Generating binary sequence $(0^{t_1}1^{s_1})(0^{t_2}1^{s_2})\cdots (0^{t_h}1^{s_h})$.
\item Calculate $T$ and $S$. $T=\sum_{i=1}^h t_i$ and   $S=\sum_{i=1}^h s_i$.
\item If $T\neq S\lor(T=S\land 
(t_1<s_1+1\lor s_h<t_h+1))$ then RETURN FALSE.
\item If  the inequalities  $\sum_{i=j}^h s_i\geq \sum_{i=j}^h t_i+1$ hold for $j=2,3,\ldots, h$ then RETURN TRUE. Otherwise, RETURN FALSE. 
\end{enumerate}

\medskip

It is not difficult to deduce that both algorithms are linear. Indeed, for step (5) of the former one and step (3) of the latter one we first compute and compare the sums for $j=h$, then for $j=h-1$, and so on. In this way, every sum is computed on the basis of the previous one and we have at most $n$ iterations such that each one is performed with $O(1)$ operations, which gives $O(n)$ for these steps. The complexity of the remaining steps is obvious.

We give some examples illustrating the applications of the  algorithms.

\begin{example}
Let $G$ be a complete split graph, i.e, a threshold graph generated by $(0^{t_1}1^{s_1})$. If $ t_1=1$, then $G$ is Hamiltonian. Otherwise, for $t_1\geq 2$, if $s_1\leq t_1-1$, then $G$ is not Hamiltonian. If $s_1\geq t_1$, then $G$ is Hamiltonian, since in this case we have $\ell=0$ and $\ell+1=h=1$.

Therefore, we conclude that a complete split graph is Hamiltonian if and only if the size of the clique is greater than or equal to the size of the co-clique, except for the case where both are equal to $1$.
\end{example}

\begin{example}
Let $G$ be a threshold graph generated by $(0^{t_1}1^{s_1}0^{t_2}1^{s_2})$. If $r=1$ (i.e., $t_1=t_2=1$), then $G$ is Hamiltonian if and only if $s_2\geq 2$. For $r\geq 2$, if either $t_1+t_2+2>s_1+s_2$, $t_1\neq 1$ or $s_2<1+t_2$, $t_1=1$, then $G$ is not Hamiltonian. Otherwise, if $s-r-1<s_1$, then $G$ is Hamiltonian if and only if $s_2\geq t_2$, while if $s-r-1\geq s_1$, then $G$ is necessarily Hamiltonian.
\end{example}

\begin{example}
Let $G$ be a particular threshold graph generated by $(0^31^40^{10}1^60^51^{11}0^31^8)$. In this case we have $r=20$ and $s=30$. Implementing the step (4) of Algorithm 1, we get $4<30-20-1<10$, which implies that $\ell=1$. We next verify that the following inequalities hold: $s_3+s_4\geq t_3+t_4$ and $s_4\geq t_4$, and since they do, we conclude that $G$ is Hamiltonian.
\end{example}

\begin{example}
Let $G$ be a particular chain graph generated by $(0^31^40^{10}1^60^51^{3}0^31^8)$. In this case we have $r=s=21$. By the step (3) of the algorithm 2, we get $s_2+s_3+s_4\not\ge t_2+t_3+t_4+1$, which implies that  $G$ is not Hamiltonian.
\end{example}

\section{The minimum number of Hamilton cycles in a Hamiltonian chain graph of a prescribed order}\label{min}

\noindent In this section we give some observations on Hamiltonian chain graphs and we also determine chain graphs with minimum number of Hamilton cycles. The problem on the  value of the minimum number of Hamilton cycles in a given graph has been considered for some special graph classes.  For existing literature and recent results related to  threshold graphs, we refer the reader to \cite{JGT}. 

An edge of a chain graph $G$ generated by \eqref{eq:bsequence} is called a {\it key edge} of $G$ if it joins a vertex in $U_i$ to a vertex in $V_{i}$ for some $1\leq i\leq h$ (see \ref{fig:nested} (b)). As it will be shown in the sequel, key edges play a significant role in determining Hamiltonian chain graphs. We proceed by the following  lemmas.

\begin{lemma}
Let $e$ be a key edge of a chain graph $G$ generated by \eqref{eq:bsequence}, then $G-e$ is a chain graph.
\end{lemma}

\begin{proof}
Let $e=uv$, where $u\in U_i$ and $v\in V_{i}$. We consider the following cases.\\

\noindent Case 1. If $t_i>1$, $s_i>1$, then $G-e$ is a chain graph generated  by $$(0^{t_1}1^{s_1})\cdots (0^{t_{i-1}}1^{s_{i-1}})(0^{t_i-1}1^1)(0^11^{s_i-1})( 0^{t_{i+1}}1^{s_{i+1}})\cdots (0^{t_h}1^{s_h}).$$

\item[] Case 2. If $t_i>1$, $s_i=1$, i.e., if $G$ is generated by $$(0^{t_1}1^{s_1})\cdots (0^{t_{i-1}}1^{s_{i-1}})(0^{t_i}1^1)(0^{t_{i+1}}1^{s_{i+1}})\cdots (0^{t_h}1^{s_h}),$$
then $G-e$ is a chain graph generated by
$$(0^{t_1}1^{s_1})\cdots (0^{t_{i-1}}1^{s_{i-1}})(0^{t_i-1}1^1)(0^{t_{i+1}+1}1^{s_{i+1}})\cdots (0^{t_h}1^{s_h}).$$

\item[]Case 3. If $t_i=1$, $s_i>1$, i.e, if $G$ is generated by $$(0^{t_1}1^{s_1})\cdots (0^{t_{i-1}}1^{s_{i-1}})(0^{1}1^{s_i})(0^{t_{i+1}}1^{s_{i+1}})\cdots (0^{t_h}1^{s_h}),$$
then $G-e$ is a chain graph generated by
$$(0^{t_1}1^{s_1})\cdots (0^{t_{i-1}}1^{s_{i-1}+1})(0^{1}1^{s_i-1})(0^{t_{i+1}+1}1^{s_{i+1}})\cdots (0^{t_h}1^{s_h}).$$

\item[] Case 4. If $t_i=s_i=1$, i.e., if $G$ is generated by $$(0^{t_1}1^{s_1})\cdots (0^{t_{i-1}}1^{s_{i-1}})(0^{1}1^1)(0^{t_{i+1}}1^{s_{i+1}})\cdots (0^{t_h}1^{s_h}),$$
then $G-e$ is a chain graph generated by
$$(0^{t_1}1^{s_1})\cdots (0^{t_{i-1}}1^{s_{i-1}+1})(0^{t_{i+1}+1}1^{s_{i+1}})\cdots (0^{t_h}1^{s_h}).$$
\end{proof}

\begin{lemma}\label{key}
Every key edge of a Hamiltonian chain graph lies in at least one Hamilton cycle.
\end{lemma}

\begin{proof}
Let $G$ be Hamiltonian chain graph, $e=uv$ be a key edge of $G$, with $u\in U_i$ and $v\in V_i$, and let $C$ be a Hamilton cycle. If $e\in C$, there is nothing to prove.  Otherwise, let $C$ has the form $(u,s,\ldots, v,t\ldots)$. Then $u\sim s$ and $v\sim t$ which implies $s\in V_j$ for some $j\geq i$ and $v\in U_k$ for some $k\leq i$ and consequently $s\sim t$. So, both $uv$ and $st$ are the edges of $G$.  The cycle $C'$ obtained by adding these two edges to $C$ and deleting $us$ and $vt$ from $C$ is Hamilton and contains $e$.
\end{proof}

We are ready for the main result of this section.

\begin{theorem}
The minimum number of  Hamilton cycles in a Hamiltonian chain   graph of order $n=2h, h\geq 2$, is $2^{h-2}$ and this number is attained uniquely by the chain graph $G_n$ generated by \begin{equation}\label{minHC}\underbrace{(0^21)(01)\ldots (01^2)}_{2(h-1)}.
\end{equation}
\end{theorem}

\begin{proof}
If $G$ is Hamiltonian chain graph, then $G$ has colour classes of the same order, say $h$ (see \cite{bip}). If $G$ is generated by $0^{t_1}1^{s_1}\cdots 0^{t_k}1^{s_k}$, then $t_1\neq 1$ and $s_k\neq 1$. The chain graph with minimum number of Hamilton cycles is defined with minimum values of $t_i$'s, $s_i$'s that, according to Corollary \ref{hcg}, are $t_1=s_k=2$, $t_i=1, i\neq 1$, $s_j=1, j\neq k$. The graph under consideration, i.e., the graph generated by \eqref{minHC}  is Hamiltonian, which is an easy exercise to prove. If any of $t_i$'s, $s_i$'s  takes a greater value than the given one, then by deleting any key edge (which by Lemma~\ref{key} belongs to at least one Hamilton cycle), we would obtain a graph that, in case that it is Hamiltonian, would have fewer number of Hamilton cycles (as the deletion of an edge cannot increase the number of  Hamilton cycles).

It remains to compute the number of Hamilton cycles, say $c_{2h}$, which can be performed by induction on $h$. If $h=2$, then 
$G_4=C_4$ and $c_4=2^{2-2}=1$.

Let $n=2(h+1)$. If $U_1=\{x,y\}$ and $V_1=\{z\}$, then by Lemma~\ref{key}, neither $G_{2(h+1)}- xz$ nor $G_{2(h+1)}-yz$ is Hamiltonian. Thus the path $xzy$ must lie in every Hamilton  cycle of $G_{2(h+1)}$ and so every Hamilton cycle of $G_{2(h+1)}$ must go through $z$. And from $z$ it may continue either through $x$ or $y$. Assume, without loss of generality, that $z$ is followed by $x$. The remaining part of the Hamilton cycle must continue through a vertex $a\notin\{ x,y,z\}$ and before it returns to $z$ it should go through $y$. Since $G_{2(h+1)}\setminus \{xz\}$ is isomorphic to $G_{2h}$, together with $2$ possible choices starting from $z$ we obtain $c_{2(h+1)}=2c_{2h}=2^{h-1}$. This completes the proof. 
\end{proof}

\section{Conclusions and future work}\label{con}

\noindent It well known that the problem of deciding whether a graph is Hamiltonian is NP-complete \cite[Chapter~8]{Sch}. In this paper we showed that for some graphs with a particular structure,  such are threshold and chain graphs, the same decision can be obtained by employing very fast algorithms. 
We also determined the minimum number of Hamilton cycles in Hamiltonian chain graphs. 

Presented results can be extended in several directions. First, by considering more general classes of graphs. A natural step after threshold graphs are the so-called cographs. By definition a graph is a cograph if it does not contain the path $P_4$ as an induced subgraph. Evidently, every threshold graph is a cograph. It is known that every cograph can be obtained by the iterative procedure based on the fact that if two graphs are cographs then their disjoint union and their join are cographs, as well \cite{StaM}. Such a procedure can be seen as an extension of the  procedure that generates threshold graphs (defined in the opening section and frequently used in this paper).

Next, one may consider the constructions of algorithms that would determine some other structural properties of threshold, chain or some similar graphs. The question of the minimum number of Hamilton cycles in a cograph is also an open problem.

Finally, our algorithms presented in Section~\ref{alg} work in general cases and both employ checking of a sequence of inequalities. It would be interesting to see under which conditions some of these inequalities can be avoided, i.e., what is the structure of the corresponding threshold or chain graph for which this part of the corresponding algorithm can be simplified.

\section*{Acknowledgement}\noindent 
We would like to thank anonymous referees for their careful reading.  Their  suggestions and observations improved the  content of the paper.

\end{document}